\documentclass[a4paper, 11pt]{amsart} 
\usepackage[
  margin=1.9cm,
  includefoot,
  footskip=30pt,
]{geometry}
\usepackage{amsmath,amssymb,amscd}
\usepackage{amsfonts}
\usepackage[english]{babel}
\usepackage[leqno]{amsmath}
\usepackage{amssymb,amsthm}
\usepackage{mathrsfs} 
\usepackage{amscd}
\usepackage{enumerate}
\usepackage{epsfig}
\usepackage{relsize}
\usepackage{layout}
\usepackage[usenames,dvipsnames]{xcolor}
\usepackage[backref=page]{hyperref}
\usepackage{tikz}
\hypersetup{
 colorlinks,
 citecolor=Green,
 linkcolor=Red,
 urlcolor=Blue}
\usepackage[matrix,arrow,tips,curve]{xy}
\input{xy}
\xyoption{all}
\definecolor{darkgreen}{rgb}{0.0, 0.7, 0.0}
\definecolor{purple}{rgb}{0.5, 0.0, 0.5}
\definecolor{red}{rgb}{0.8, 0.2, 0.0}

\newtheorem{thm}{Theorem}[section]
\newtheorem{bthm}{Theorem}

\newtheorem{lemma}[thm]{Lemma}

\newtheorem{claim}[thm]{Claim}

\numberwithin{equation}{section}
\theoremstyle{definition}
\newtheorem{defi}[thm]{Definition}

\theoremstyle{remark}
\newtheorem{remark}[thm]{Remark}

\newcommand{\Z}{\mathbb{Z}}
\newcommand{\C}{\mathbb{C}}
\newcommand{\Q}{\mathbb{Q}}

\def \P{\mathbb{P}}

\def \F{\mathcal F}

\def \E{\mathcal E}

\def\O{\mathcal O}

\def\M0{\mathcal M^0}

\newcommand{\ch}{\operatorname{ch}}
\newcommand{\td}{\operatorname{td}}
\begin{document}

\title[A linear lower bound on the Ulrich complexity of hypersurfaces]{A linear lower bound on the Ulrich complexity of hypersurfaces}

\author[A.F. Lopez, D. Raychaudhury, Y. Takahashi]{Angelo Felice Lopez, Debaditya Raychaudhury and Yuta Takahashi}

\address{\hskip -.43cm Angelo Felice Lopez, Dipartimento di Matematica e Fisica, Universit\`a di Roma
Tre, Largo San Leonardo Murialdo 1, 00146, Roma, Italy. e-mail {\tt angelo.lopez@uniroma3.it}}

\address{\hskip -.43cm Debaditya Raychaudhury, Department of Mathematics and Statistics, University of New Mexico, Albuquerque, NM 87131. e-mail {\tt rcdeba@gmail.com}}

\address{\hskip -.43cm Yuta Takahashi, Department of Mathematics, Faculty of Science and Engineering, Chuo University. 1-13-27 Kasuga, Bunkyo-ku, Tokyo 112-8551, Japan. email: {\tt yuta0630takahashi0302@gmail.com}}

\thanks{The first author is partially supported by the GNSAGA group of INdAM. The third author is partially supported by JST SPRING, Japan Grant Number JPMJSP2170.}


\begin{abstract} 
We give a lower bound on the Ulrich complexity of hypersurfaces in terms of their dimension.
\end{abstract}

\maketitle

\section{Introduction}

Let $X \subset \P^N$ be a smooth irreducible variety of dimension $n$. In order to study the geometry of $X$, in recent years, an interesting point of view has emerged, the one of Ulrich bundles. An Ulrich bundle is a vector bundle $\E$ on $X$ such that $H^i(\E(-p))=0$ for all $i \ge 0$ and $1 \le p \le n$. In the presence of such bundles, several geometrical features can be discovered on $X$, see for example the seminal paper \cite{es}, the survey \cite{b} and the book \cite{cmrpl}. The main open problem about Ulrich bundles is their existence, conjectured to happen in all cases. Aside for curves and several surfaces, there are a few families of varieties that are known to  carry an Ulrich bundle, among which complete intersections, by \cite{hub}. Once the existence is established, one defines the Ulrich complexity as
$${\rm uc}(X, \O_X(1))=\min\{r \ge 1 : \ \hbox{there exists a rank} \ r \ \hbox{Ulrich bundle on} \ X\}.$$ 
We simply write ${\rm uc}(X)$ when $\O_X(1)$ is naturally given.

In the case of a smooth hypersurface $X \subset \P^{n+1}$ of degree $d \ge 2$ (when $d=1$, ${\rm uc}(X)=1$), Ulrich complexity of $X$ falls within the Buchweitz, Greuel and Schreyer's conjecture \cite{bgs}, a stronger conjecture regarding aCM bundles, implying that ${\rm uc}(X) \ge 2^{\lfloor \frac{n-1}{2} \rfloor}$ (sharp for $d=2$). It is also conjectured in \cite{rt1} that ${\rm uc}(X) \ge 2^{\lfloor \frac{n+1}{2} \rfloor}$ when $X$ is general. The above conjectures are wide open, the only known general result, of a slightly different flavour, being \cite{e}. As far as we know, the best-known lower bound in terms of the dimension was shown in \cite[Thm.~3.1]{bes}: 
$${\rm uc}(X) \ge \sqrt{n+2}-1.$$
 
It is the purpose of this paper to improve the above bound to a lower bound that is essentially the dimension. In order to state it precisely, let us define, for integers $m \ge 3, d \ge 3$, the functions
$$B(m)=\frac{30\pi^2(2m-1)(2m-3)}{24d^2(2m-3)\sin^2(\frac{\pi}{d})+5\pi d(2m-1)\sin(\frac{\pi}{d})}$$
and
$$F(m)= \min\{2m, B(m)\}.$$
Observe that $F(m)=2m$ for $m \ge 21$, see Section \ref{num}. It is also possible that, for $3 \le m \le 20$, still $F(m)=2m$. The precise values of $d$ for which the latter holds are given in Table \ref{tab1}. For the other values of $d$, note that $F(m) \ge \frac{3m}{2}$ for $3 \le m \le  20$. Then we have

\begin{bthm} \hskip 3cm
\label{main}

Let $X \subset \P^{n+1}$ be a smooth hypersurface of dimension $n \ge 6$ and degree $d \ge 3$. If $n$ is odd we have 
$${\rm uc}(X) \ge \begin{cases} n-1 & \text{ if } n \ge 43 \\ F(\frac{n-1}{2}) & \text{ if } 7 \le n \le 41
\end{cases}$$
while if $n$ is even we have
$${\rm uc}(X) \ge \begin{cases} n-2 & \text{ if } n \ge 44 \\ F(\frac{n-2}{2}) & \text{ if } 6 \le n \le 42.
\end{cases}$$
Moreover, if $n$ is even and $X$ is very general, then
$${\rm uc}(X) \ge \begin{cases} n & \text{ if } n \ge 44 \\ F(\frac{n}{2}) & \text{ if } n \le 42.
\end{cases}.$$
\end{bthm}

Similar calculations can be performed for $n \in \{4, 5\}$, but they give ${\rm uc}(X) \ge 4$, a result that is already known (see for example \cite{lr1, rt1, rt2} where several lower bounds are given).

Note that $F(m)=2m$ for any $m \ge 3$ when $d=3$ by Table \ref{tab1}. Hence, if $X$ is a smooth cubic hypersurface of dimension $n \ge 6$, then ${\rm uc}(X) \ge n-1$ if $n$ is odd, ${\rm uc}(X) \ge n-2$ if $n$ is even and ${\rm uc}(X) \ge n$ if $n$ is even and $X$ is very general. On the other hand, since $r$ is divisible by $3$ for  any rank $r$ Ulrich bundle on $X$ \cite[Prop. 2.5]{ks}, \cite{fk}, the above lower bounds improve to the first multiple of $3$.

We work over the complex numbers.

\section{Homogeneous symmetric polynomials and series}

In order to study the Chern character of an Ulrich bundle, we will need a few elementary estimates.

\begin{defi}
For each $n \ge 1, N \ge 1$ we denote by 
$$h_n(X_1,\ldots, X_N)= \sum\limits_{\substack{i_1+\ldots+i_N=n \\ i_1 \ge 0 \ldots i_N \ge 0}} X_1^{i_1} \cdots X_N^{i_N}$$ 
the {\it complete homogeneous symmetric polynomial of degree $n$ in the variables $X_1,\ldots, X_N$}. We set $h_0=1$.
\end{defi}

Then we have

\begin{lemma}
\label{els}
For $i \ge 1$, let $x_i \in \C$ be such that $\sum\limits_{i \ge 1} x_i$ is absolutely convergent and let $M>0$ be such that $|x_i| \le M$ for all $i \ge 1$. Let
$$G(z)=\prod_{i \ge 1}(1-x_iz)^{-1}.$$
Then $G(z)$ is well-defined for $|z| \le R$ and any $R$ with $0 < R < \frac{1}{M}$. Moreover
\begin{equation}
\label{mac}
G(z)=\sum_{n \ge 0} h_n(x_1,x_2,\ldots)z^n
\end{equation}
where $h_n(x_1,x_2,\ldots)=\sum\limits_{\substack{m_1+m_2+\ldots = n \\ m_1 \ge 0, m_2 \ge 0, \ldots}} \prod\limits_{i \ge 1} x_i^{m_i}$ is also absolutely convergent.
\end{lemma}
\begin{proof}
Since $\sum_{i \ge 1} x_i$ is absolutely convergent, we have that $|x_i| \to 0$ for $i \to \infty$, hence an $M>0$ such that $|x_i| \le M$ for all $i \ge 1$ exists. If $|z| \le R$, we have that $|1-x_iz| \ge 1-|x_i||z| \ge 1-MR>0$ and therefore
\begin{equation}
\label{unif}
\left|\frac{1}{1-x_iz}-1\right|=\frac{|x_i||z|}{|1-x_iz|} \le \frac{|x_i|R}{1-MR}.
\end{equation}
Since $(1-x_iz)^{-1} \ne 0$ for all $i \ge 1$, it follows by \cite[Chapt.~5, Thm.~6]{ahl} that the product $G(z)=\prod_{i \ge 1}(1-x_iz)^{-1}$ is absolutely convergent, hence well-defined. Also, we get that
$$G(z) = \lim\limits_{N \to \infty} G_N(z)$$
where 
$$G_N(z)=\prod_{i \ge 1}^N(1-x_iz)^{-1}= \sum\limits_{n \ge 0}h_n(x_1,\ldots, x_N)z^n.$$
Moreover, the Weierstrass M-test and \eqref{unif} imply that $G_N(z)$ converges uniformly to $G(z)$ for $|z| \le R$.
Now note that the partial sums of the series $h_n(x_1,x_2,\ldots)$ are exactly the $h_n(x_1,\ldots, x_N)$ and
$$|h_n(x_1,\ldots, x_N)| \le \sum\limits_{\substack{i_1+\ldots+i_N=n \\ i_1 \ge 0 \ldots i_N \ge 0}} |x_1|^{i_1} \cdots |x_N|^{i_N}\le \left(\sum\limits_{i=1}^N|x_i|\right)^n.$$
Since all sums $\sum\limits_{i=1}^N|x_i|$ are bounded by hypothesis, it follows that $h_n(x_1,x_2,\ldots)$ is absolutely convergent and
$$\lim\limits_{N \to \infty} h_n(x_1,\ldots, x_N)=h_n(x_1,x_2,\ldots).$$
Finally, Cauchy's formula gives that the coefficient of $z^n$ in the Maclaurin series of $G(z)$ is precisely the limit, for $N \to \infty$, of the coefficients of $z^n$ in the Maclaurin series of $G_N(z)$ and this proves \eqref{mac}.
\end{proof}

We now consider the function that will calculate the Chern character of an Ulrich bundle.

\begin{lemma}
\label{mcl}
Let $d \in \Z, d \ge 3$, let $u=\frac{d-1}{2}$, let
$$f(z)=\frac{d}{\sum\limits_{k=0}^{d-1}e^{(u-k)z}}=\frac{d}{e^{uz}+e^{(u-1)z} + \ldots+e^{-uz}}$$
and consider its Maclaurin series
\begin{equation}
\label{serie}
f(z)=1+\sum_{j \ge 1}(-1)^j b_j \frac{z^{2j}}{(2j)!}.
\end{equation}
If $q:=\frac{d^2}{4 \pi^2}$ we have, for all $j \ge 1$, that the following estimates hold:
\begin{equation}
\label{bd}
\frac{5}{4} q^j \le \frac{b_j}{(2j)!}\le \frac{2\sin(\frac{\pi}{d})}{\frac{\pi}{d}} q^j.
\end{equation}
\end{lemma}
\begin{proof}
Note that $f(z)$ is holomorphic in a neighborhood of $0$ and $f(0)=1$. Moreover $f(z)$ is an even function, hence \eqref{serie} makes sense. By the identity
$$e^{uz}+e^{(u-1)z} + \ldots+e^{-uz}=\sum\limits_{k=0}^{d-1}e^{(u-k)z}=\frac{\sinh(\frac{dz}{2})}{\sinh(\frac{z}{2})}$$
we get that
\begin{equation}
\label{3}
f(z)=\frac{d\sinh(\frac{z}{2})}{\sinh(\frac{dz}{2})}.
\end{equation}
By Euler's product formula (that can be obtained from \cite[(24), Chapt. 5, \S 2.3]{ahl} by replacing $\sinh(z)=-i\sin(iz)$),
$$\sinh(z)=z\prod\limits_{\ell \ge 1}\left(1+\frac{z^2}{\pi^2 \ell^2}\right)$$
we get from \eqref{3}, that
\begin{equation}
\label{4}
f(z)=\frac{d\sinh(\frac{z}{2})}{\sinh(\frac{dz}{2})}=\frac{\frac{dz}{2}\prod\limits_{\ell \ge 1}\left(1+\frac{z^2}{4\pi^2\ell^2}\right)}{\frac{dz}{2}\prod\limits_{\ell \ge 1}\left(1+\frac{d^2z^2}{4\pi^2\ell^2}\right)}=\prod\limits_{\substack{\ell \ge 1 \\ d \nmid \ell}}\left(1+\frac{d^2z^2}{4\pi^2\ell^2}\right)^{-1}.
\end{equation}
Now, it follows by Lemma \ref{els} that
\begin{equation}
\label{5}
\frac{b_j}{(2j)!}=q^j h_j
\end{equation}
where $h_j:=h_j(\frac{1}{\ell^2}, d \nmid \ell, \ell \ge 1, \ldots)$ is the series of degree $j$ in the positive variables $\frac{1}{\ell^2}, \ell \ge 1, d \nmid \ell$.

For the lower bound in \eqref{bd}, just observe that $h_j \ge 1+\frac{1}{4}=\frac{5}{4}$, which gives, using \eqref{5}, that $\frac{b_j}{(2j)!}\ge \frac{5}{4} q^j$.

For the upper bound in \eqref{bd}, using \cite[Chapt. 5, Ex.1, \S 2.2 and (24), \S 2.3]{ahl}, we have
$$\frac{1}{2}=\prod_{\ell \ge2} \left(1-\frac{1}{\ell^2}\right)=\prod_{\substack{\ell\ge2 \\ d \nmid \ell}} \left(1-\frac{1}{\ell^2}\right)\prod_{\ell' \ge 1} \left(1-\frac{1}{d^2(\ell')^2}\right)=\prod_{\substack{\ell\ge2 \\ d \nmid \ell}} \left(1-\frac{1}{\ell^2}\right) \frac{\sin(\frac{\pi}{d})}{\frac{\pi}{d}}$$
hence
\begin{equation}
\label{6}
\prod_{\substack{\ell\ge2 \\ d \nmid \ell}} \left(1-\frac{1}{\ell^2}\right)=\frac{\pi}{2d\sin(\frac{\pi}{d})}.
\end{equation}
Set now $x_{\ell}=\frac{1}{\ell^2}$ for $\ell \ge 2$ with $d \nmid \ell$ in Lemma \ref{els}. Since $\frac{1}{\ell^2} \le \frac{1}{4}$, we can set $M=\frac{1}{4}$ in Lemma \ref{els} and therefore the convergence holds, for example, for $|z| \le 3$. Then \eqref{mac} gives
\begin{equation}
\label{7}
\prod_{\substack{\ell\ge2 \\ d \nmid \ell}} \left(1-\frac{1}{\ell^2}\right)^{-1} = \prod_{\substack{\ell\ge2 \\ d \nmid \ell}}(1-x_{\ell})^{-1}= G(1)=\sum_{n \ge 0} h_n(\frac{1}{\ell^2}, d \nmid \ell, \ell \ge 2, \ldots).
\end{equation}
On the other hand, using \eqref{6} and \eqref{7}, we see that
$$h_j \le \sum_{n \ge 0} h_n(\frac{1}{\ell^2}, d \nmid \ell, \ell \ge 2, \ldots)=\prod_{\substack{\ell\ge2 \\ d \nmid \ell}} \left(1-\frac{1}{\ell^2}\right)^{-1}=\frac{2d\sin(\frac{\pi}{d})}{\pi}.$$
Therefore \eqref{5} gives that
$$\frac{b_j}{(2j)!}=q^j h_j \le \frac{2\sin(\frac{\pi}{d})}{\frac{\pi}{d}} q^j$$
and the lemma is proved.
\end{proof}

\section{$\Q$-twisted vector bundles and their Chern character}

We recall the notion of $\Q$-twisted vector bundle (see for example \cite[\S 6.2 and 8.1]{laz}). In the definition below we use the well-known notation $\binom{\ell}{m}=\frac{\ell (\ell-1)\ldots (\ell-m+1)}{m!} \ \mbox{for} \ m, \ell \in \Z, m \ge 1$.

\begin{defi}
Let $X$ be a smooth projective variety of dimension $n$. A {\it $\Q$-twisted vector bundle} $\E\langle\delta\rangle$ on $X$ is a pair $(\E, \delta)$ where $\E$ is a rank $r$ bundle on $X$ and $\delta \in N^1(X)_{\Q}$.

The Chern classes of $\E\langle\delta\rangle$ are
$$c_i(\E\langle\delta\rangle)=\sum\limits_{k=0}^i \binom{r-k}{i-k}c_k(\E)\delta^{i-k} \in H^{2i}(X; \Q)$$
where we view $\delta \in H^2(X; \Q)$.

The Chern character of $\E\langle\delta\rangle$ is $ch(\E\langle\delta\rangle)=ch(\E)e^{\delta}$. For $0 \le k \le n$, we denote by $ch_k(\E\langle\delta\rangle)$ its degree $k$ part.
\end{defi}

\begin{remark}
\label{dopor}
The definition of Chern classes of $\E\langle\delta\rangle$ \cite[Def.~8.1.1]{laz} comes of course from the usual formula for $c_i(\E \otimes L)$, where $L$ is a line bundle. Moreover, as in the case of a line bundle, we observe that $c_i(\E\langle\delta\rangle)=0$ when $i>r$: Indeed 
$$c_i(\E\langle\delta\rangle)=\sum\limits_{k=0}^r \binom{r-k}{i-k}c_k(\E)\delta^{i-k}+\sum\limits_{k=r+1}^i \binom{r-k}{i-k}c_k(\E)\delta^{i-k}.$$
Now, in the second sum we have that $c_k(\E)=0$, while in the first sum we have that 
$$\binom{r-k}{i-k}=\frac{(r-k)(r-k-1)\ldots (r-i+1)}{(i-k)!}=0.$$
\end{remark}
%\section{Newton powers and Chern character}

We now give a simple result about Newton power sums of $\Q$-twisted vector bundles.

\begin{defi}
Let $X$ be a smooth irreducible variety and let $\E\langle\delta\rangle$ be a $\Q$-twisted vector bundle on $X$. For each $k \ge 0$, {\it the $k$-th Newton power sum of $\E\langle\delta\rangle$} is $p_k(\E\langle\delta\rangle) = k! \ch_k(\E\langle\delta\rangle)$.
\end{defi}

We have the following

\begin{lemma}
Let $X$ be a smooth irreducible variety and let $\E\langle\delta\rangle$ be a rank $r$ $\Q$-twisted vector bundle on $X$. Then, for every $k \ge 1$, 
\begin{equation}
\label{9}
k\,c_k(\E\langle\delta\rangle) = \sum_{i=1}^k(-1)^{i-1}p_i(\E\langle\delta\rangle)c_{k-i}(\E\langle\delta\rangle)
\end{equation}
holds in $H^{2k}(X;\Q)$.
\end{lemma}

\begin{proof}
By the splitting principle, it is enough to prove the identity after pulling back to a space on which $\E$ splits as a direct sum of line bundles. If $x_1,\dots, x_r$ are the Chern roots of $\E$, then, as in \cite[Proof of Lemma 8.1.2]{laz}, the Chern roots of $\E\langle\delta\rangle$ are $y_i:=x_i+\delta, 1 \le i \le r$. Therefore
$$c_k(\E\langle\delta\rangle)=e_k(y_1,\dots,y_r)$$
where $e_k$ is the $k$-th elementary symmetric polynomial, and
$$p_i(\E\langle\delta\rangle)=y_1^i+\cdots + y_r^i.$$
Indeed,
$$\ch(\E\langle\delta\rangle)=ch(\E)e^{\delta}=\big(\sum_{j=1}^r e^{x_j}\big)e^{\delta}=\sum_{j=1}^r e^{y_j} = \sum_{j=1}^r\sum_{i \ge 0}\frac{y_j^i}{i!}$$
so that $p_i(\E\langle\delta\rangle)=i! \ch_i(\E\langle\delta\rangle)=\sum_{j=1}^r y_j^i$. Now set
$$C(t)=\prod_{j=1}^r(1+y_j t)=\sum_{k\ge 0}c_k(\E\langle\delta\rangle)t^k$$
where the sum is clearly a finite sum. Then
$$\frac{C'(t)}{C(t)}=\sum_{j=1}^r\frac{y_j}{1+y_j t}.$$
Expanding each summand and observing that all sums below are finite, gives
$$\frac{y_j}{1+y_j t}=\sum_{a\ge 0}(-1)^a y_j^{a+1}t^a.$$
Hence
$$\frac{C'(t)}{C(t)}=\sum_{a\ge 0}(-1)^a\left(\sum_{j=1}^r y_j^{a+1}\right)t^a=\sum_{i \ge 1}(-1)^{i-1}p_i(\E\langle\delta\rangle)t^{i-1}.$$
Multiplying by $tC(t)$, we obtain
$$tC'(t)=C(t) \sum_{i \ge 1}(-1)^{i-1}p_i(\E\langle\delta\rangle)t^i.$$
Now
$$tC'(t)=\sum_{k\ge 1}k\,c_k(\E\langle\delta\rangle)t^k$$
while
$$C(t)\sum_{i \ge 1}(-1)^{i-1}p_i(\E\langle\delta\rangle)t^i=\left(\sum_{a\ge 0}c_a(\E\langle\delta\rangle)t^a\right)
\left(\sum_{i \ge 1}(-1)^{i-1}p_i(\E\langle\delta\rangle)t^i \right).$$
Comparing the coefficient of $t^k$, we get
$$k\,c_k(\E\langle\delta\rangle)=\sum_{i=1}^k(-1)^{i-1}p_i(\E\langle\delta\rangle)c_{k-i}(\E\langle\delta\rangle).$$
\end{proof}

\section{Proof of Theorem \ref{main}}

We first calculate the Chern character of an Ulrich vector bundle on a hypersurface.

\begin{lemma}
\label{ch}
Let $X \subset \P^{n+1}$ be a smooth hypersurface of dimension $n \ge 3$ and degree $d \ge 3$. Assume that either $n$ is odd, or $n$ is even and $X$ is very general. Let $\E$ be a rank $r$ Ulrich bundle on $X$. Then, if $H \in H^2(X; \Q)$ is the class of a hyperplane section, we have
$$\ch(\E)=\frac{rd}{1+e^{-H}+\ldots+e^{-(d-1)H}}.$$
\end{lemma}
\begin{proof}
Let $ i : X \hookrightarrow \P^{n+1}$ be the inclusion. Let $h \in H^2(\P^{n+1};\Q)$ be the class of a hyperplane, so that $i^*h=H$. As is well known (see for example \cite[Prop.~2.1(ii)]{b}), the sheaf $i_*\E$ has a linear resolution on $\P^{n+1}$ of the form
$$0 \to \O_{\P^{n+1}}(-1)^{\oplus (rd)} \to \O_{\P^{n+1}}^{\oplus (rd)} \to i_*\E \to 0$$
hence $\ch(i_*\E)=rd(1-e^{-h}) \in H^*(\P^{n+1};\Q)$. Using the latter and Grothendieck--Riemann--Roch's theorem for a closed embedding \cite[\S 15.2, page 288]{f}, we find that
\begin{equation}
\label{1}
rd(1-e^{-h})=\ch(i_*\E)= i_*\left(\td(\O_X(dH))^{-1} \cdot \ch(\E)\right)=i_*\left(\frac{1-e^{-dH}}{dH} \cdot \ch(\E)\right).
\end{equation}
Under the given hypotheses on $X$, we have that $i^*: H^{2k}(\P^{n+1}; \Q) \cong \Q h^k \to H^{2k}(X; \Q) \cong \Q H^k$ is an isomorphism for $k \ne \frac{n}{2}$ (in particular for odd $n$), while for even $n$ and $k =\frac{n}{2}$, any algebraic class in $H^n(X; \Q)$ is of type $aH^{\frac{n}{2}}$ (see for example \cite[Lemma 4.1]{lr2}), hence in particular so is $\ch_{\frac{n}{2}}(\E)$. Therefore, there is $\beta \in H^*(\P^{n+1};\Q)$ such that $\ch(\E)=i^*\beta$. Setting $\alpha = \frac{1-e^{-dh}}{dh}$, we have that $i^* \alpha = \frac{1-e^{-dH}}{dH}$. Therefore \eqref{1} becomes, using $i_*(1)=[X]=dh$ and the push-pull formula,
$$rd(1-e^{-h})=i_*\left(i^*\alpha \cdot i^*\beta)\right)=i_*\left(1 \cdot i^*(\alpha \cdot \beta)\right)=dh \alpha \beta= (1-e^{-dh})\beta.$$
Therefore $\beta = \frac{rd(1-e^{-h})}{1-e^{-dh}}$ and then 
$$\ch(\E)=i^*\beta = \frac{rd(1-e^{-H})}{1-e^{-dH}}=\frac{rd}{1+e^{-H}+\ldots+e^{-(d-1)H}}.$$
\end{proof}

We are now ready for the proof of Theorem \ref{main}.
\begin{proof}
Suppose to begin with that either $n$ is odd, or $n$ is even and $X$ is very general, so that  Lemma \ref{ch} applies. Now, assume that there is a rank $r$ Ulrich bundle $\E$ on $X$, set $u=\frac{d-1}{2}$ and set 
$$\F=\E\langle-uH\rangle.$$
Then, applying Lemma \ref{ch}, we have
$$\ch(\F) =\ch(\E)e^{-uH}=\frac{rde^{-uH}}{1+e^{-H}+\ldots+e^{-2uH}}$$
hence
\begin{equation}
\label{8}
\ch(\F)=\frac{rd}{e^{uH}+e^{(u-1)H} + \ldots+e^{-uH}}.
\end{equation}
Now observe that the identity
$$d=\left(e^{uz}+\ldots+e^{-uz}\right)\left(1+\sum_{j \ge 1}(-1)^j b_j \frac{z^{2j}}{(2j)!}\right)$$
holds by Lemma \ref{mcl}, hence we have that the identity
$$d=\left(e^{uH}+\ldots+e^{-uH}\right)\left(1+\sum_{j \ge 1}(-1)^j b_j \frac{H^{2j}}{(2j)!}\right)$$
also holds as product of formal power series. Therefore, setting $b_0=1$, \eqref{8} gives that
$$\ch(\F)=\sum_{j \ge 0}(-1)^j rb_j \frac{H^{2j}}{(2j)!}$$
and then we see that, for all $j \ge 0$, 
\begin{equation}
\label{10}
p_{2j+1}(\F)=0,\ \ p_{2j}(\F)=(-1)^j rb_jH^{2j}.
\end{equation}
For $0 \le k \le n$, define the intersection numbers
$$I_k =c_k(\F)H^{n-k}$$
so that, in particular, $I_0=H^n=d$. Setting $k=2s$ in \eqref{9}, intersecting with $H^{n-2s}$ and using \eqref{10}, we find, for all $0 \le s \le \frac{n}{2}$, that
\begin{equation}
\label{11}
2sI_{2s}=\sum_{i=1}^{2s}(-1)^{i-1} p_i(\F)c_{2s-i}(\F)H^{n-2s}=\sum_{j=1}^s(-1)^{j+1}rb_jI_{2s-2j}.
\end{equation}
We shall prove the following purely numerical fact.
\begin{claim}
\label{finale}
Let $m \in \Z$ be such that $3 \le m \le \frac{n}{2}$ and assume that $r < F(m)$. Then
$$(-1)^{m-1}I_{2m}>0.$$
\end{claim}
\begin{proof}
For all $0 \le s \le \frac{n}{2}$, set
\begin{equation}
\label{j}
J_s =\frac{I_{2s}}{d}
\end{equation}
so that $J_0=1$ and \eqref{11} becomes
$$2sJ_s=\sum_{j=1}^{s}(-1)^{j+1}rb_jJ_{s-j}$$
that is
\begin{equation}
\label{12}
sJ_s=\sum_{j=1}^{s}jA_jJ_{s-j},
\end{equation}
where
\begin{equation}
\label{13}
A_j:=(-1)^{j+1}\frac{rb_j}{2j}.
\end{equation}
We will now prove that
\begin{equation}
\label{14}
J_s=\sum_{\lambda\vdash s}\left(\prod_{j\ge1}\frac{A_j^{s_j}}{s_j!}\right)
\end{equation}
where $\lambda=(1^{s_1}2^{s_2}\cdots)$ runs over all partitions of $s \ge 0$ such that
$$\sum_{j\ge1}js_j=s.$$
First, note that the product on the right-hand side of \eqref{14} is finite. To see \eqref{14}, define, for every $s \ge 0$, 
$$R_0=1 \ \hbox{and} \ R_s=\sum_{\lambda\vdash s}\left(\prod_{j\ge1}\frac{A_j^{s_j}}{s_j!}\right).$$
Consider the generating function
$$R(t)=\sum\limits_{s \ge 0} R_s t^s.$$
We have
\begin{equation}
\label{coe}
R(t)  = \sum\limits_{s \ge 0} \sum\limits_{\substack{s_1 \ge 0, s_2 \ge 0, \ldots \\ \sum_{j\ge1}js_j=s}}\left(\prod_{j\ge1}\frac{A_j^{s_j}}{s_j!}t^s\right)=\sum\limits_{\substack{(s_1, s_2, \ldots), s_i \ge 0 \\ \ \hbox{with finite support}}}\left(\prod_{j\ge1}\frac{(A_jt^j)^{s_j}}{s_j!}\right).
\end{equation}
Consider the infinite product of formal power series $\prod_{j\ge1} (1+F_j(t))$
where
$$F_j(t)=\sum\limits_{s_j \ge 1}\frac{(A_jt^j)^{s_j}}{s_j!}.$$
Then $F_j(0)=0$ and $\deg F_j(t) =j$ since $A_j \ne 0$ by \eqref{13} and \eqref{bd}. Hence 
$$\lim\limits_{j \to \infty} \deg F_j(t) = \infty$$ 
and $\prod_{j\ge1} (1+F_j(t))$ converges, by \cite[Prop.~1.1.9]{st}, to a formal power series whose coefficient of $t^s$ is clearly the same as the one in \eqref{coe}. This implies that, as formal power series,
$$\sum\limits_{\substack{(s_1, s_2, \ldots), s_i \ge 0 \\ \ \hbox{with finite support}}}\left(\prod_{j\ge1}\frac{(A_jt^j)^{s_j}}{s_j!}\right) = \prod_{j\ge1}\left(\sum\limits_{s_j \ge 0}\frac{(A_jt^j)^{s_j}}{s_j!}\right)= \prod_{j\ge1} e^{A_jt^j}=e^{\sum\limits_{j \ge 1}A_jt^j}$$
where the last step is possible since if $H(t)=\sum_{j \ge 1}A_jt^j$, then $H(0)=0$, hence $e^{H(t)}$ is a well-defined composition of two series \cite[\S 1.1]{st}. 
%Consider the infinite product of formal power series $\prod_{j\ge1} (1+F_j(t))$
%where
%$$F_j(t)=\sum\limits_{\substack{(s_1, s_2, \ldots), s_i \ge 0, \ \hbox{not all} \ 0 \\ \ \hbox{with finite support}}}\frac{(A_jt^j)^{s_j}}{s_j!}.$$
%Then $F_j(0)=0$ and $\deg F_j(t) =j$, hence 
%$$\lim\limits_{j \to \infty} \deg F_j(t) = \infty.$$ 
%Therefore the infinite product $\prod_{j\ge1} (1+F_j(t))$ converges to a formal power series by \cite[Prop.~1.1.9]{st}
%We observe that, in the formal power series ring, to compute the coefficient of $t^s$ in the above product,
%only finitely many $j$'s matter, namely the ones with $j \le s$. Moreover, for each $j \le s$, only finitely many values $s_j \le \frac{s}{j}$ can contribute. Hence the coefficient of $t^s$ is obtained by a finite sum. This implies that, as formal power series,
%and we have that
%$$\sum\limits_{\substack{(s_1, s_2, \ldots), s_i \ge 0 \\ \ \hbox{with finite support}}}\left(\prod_{j\ge1}\frac{(A_jt^j)^{s_j}}{s_j!}\right) = \prod_{j\ge1}\left(\sum\limits_{\substack{(s_1, s_2, \ldots), s_i \ge 0 \\ \ \hbox{with finite support}}}\frac{(A_jt^j)^{s_j}}{s_j!}\right)= \prod_{j\ge1} e^{A_jt^j}=e^{\sum\limits_{j \ge 1}A_jt^j}$$
%where the last step is possible since if $H(t)=\sum_{j \ge 1}A_jt^j$, then $H(0)=0$, hence $e^{H(t)}$ is a well-defined composition of two series \cite[\S 1.1]{st}. 
Then
$$R(t)=e^{\sum\limits_{j \ge 1}A_jt^j}$$
and therefore
$$R'(t)=R(t)\sum\limits_{j \ge 1}jA_jt^{j-1}$$
and
$$tR'(t)=R(t)\sum\limits_{j \ge 1}jA_jt^j.$$
Comparing the coefficient of $t^s$, we have, on the left-hand side $sR_s$ and, on the right-hand side,
$$\sum_{\substack{k+j=s \\ k \ge 0, j \ge 1}} R_kjA_j=\sum_{j=1}^s jA_jR_{s-j}.$$
Thus, $R_s$ satisfies the same recursion \eqref{12} and initial value as $J_s$ and therefore $R_s=J_s$, that is \eqref{14} holds.
 
Next, consider \eqref{14} for $s=m$, that is
\begin{equation}
\label{15}
J_m=\sum_{\lambda\vdash m}\prod_{j \ge 1}\frac{A_j^{m_j}}{m_j!}.
\end{equation}
The term corresponding to the partition $\lambda=(1^0, 2^0, \ldots, m^1)$ in the above is, by \eqref{13},
$$A_m=(-1)^{m+1}\frac{rb_m}{2m}.$$
Let
$$M:=\frac{rb_m}{2m}$$
so that, using \eqref{bd},
\begin{equation}
\label{24}
(-1)^{m-1}A_m=M>0.
\end{equation}
We shall prove that the sum of the absolute values of all remaining terms in \eqref{15} is strictly smaller than $M$. 

To this end, let $\ell(\lambda)=\sum_{j \ge 1}m_j$ be the length of a partition $\lambda$. For $\ell \ge2$, let 
$$T_\ell = \sum_{\substack{\lambda\vdash m \\ \ell(\lambda)=\ell}} \prod_{j \ge 1}\frac{|A_j|^{m_j}}{m_j!}$$
where we note that sum and product are finite. We get from \eqref{13} and \eqref{bd} that
\begin{equation}
\label{16}
|A_j| = \frac{rb_j}{2j}=r((2j-1)!)\frac{b_j}{(2j)!} \le r((2j-1)!)\frac{2d\sin(\frac{\pi}{d})}{\pi} q^j
\end{equation}
and
\begin{equation}
\label{17}
M=|A_m| =r((2m-1)!)\frac{b_m}{(2m)!} \ge r((2m-1)!)\frac{5}{4}q^m.
\end{equation}
Now define
$$P_{m,\ell} = \sum_{\substack{\lambda\vdash m\\ \ell(\lambda)=\ell}}\prod\limits_{j\ge1}\frac{((2j-1)!)^{m_j}}{m_j!}$$
where we note that sum and product are finite.
Using \eqref{16} and \eqref{17} we get
\begin{equation}
\label{18}
\begin{aligned}
\frac{T_\ell}{M} & = \sum_{\substack{\lambda\vdash m \\ \ell(\lambda)=\ell}} \frac{1}{M} \prod_{j \ge 1}\frac{|A_j|^{m_j}}{m_j!} \le \sum_{\substack{\lambda\vdash m \\ \ell(\lambda)=\ell}} \frac{1}{M} \prod_{j \ge 1}\frac{r^{m_j}((2j-1)!)^{m_j}\left(\frac{2d\sin(\frac{\pi}{d})}{\pi}\right)^{m_j}q^{jm_j}}{m_j!}= \\
& = \frac{r^{\ell}q^m\left(\frac{2d\sin(\frac{\pi}{d})}{\pi}\right)^{\ell}}{M} \sum_{\substack{\lambda\vdash m \\ \ell(\lambda)=\ell}} \prod_{j \ge 1}\frac{((2j-1)!)^{m_j}}{m_j!} \le \frac{r^{\ell-1} \left(\frac{2d\sin(\frac{\pi}{d})}{\pi}\right)^{\ell}}{\frac{5}{4}(2m-1)!} P_{m,\ell}.
\end{aligned}
\end{equation}
We now estimate $P_{m,\ell}$. Consider, for $\ell \le m$,
$$Q_{m, \ell}:= \frac{1}{\ell!}\sum_{\substack{i_1+\cdots+i_\ell=m \\ i_a \ge 1}}\prod_{a=1}^{\ell}(2i_a-1)!.$$
We will prove that $P_{m,\ell}=Q_{m,\ell}$. To this end, for every $\ell$-tuple $(i_1, \ldots, i_{\ell})$ such that $i_a \ge 1$ for $a \ge 1$ and $i_1+\cdots+i_\ell=m$, define, for every $j \ge 1$,
$$m_j=\#\{a \in \{1,\ldots, \ell\}: i_a=j\}$$
so that $\sum\limits_{j \ge 1}m_j=\ell, \sum\limits_{j \ge 1}jm_j=m$ and
$$\prod\limits_{a=1}^{\ell}(2i_a-1)!=\prod\limits_{j\ge1}((2j-1)!)^{m_j}.$$ 
Vice versa, for every partition $(1^{m_1}2^{m_2}\cdots)$ of $m$ of length $\ell$, there are $\frac{\ell!}{\prod_{j \ge 1}m_j!}$ possibilities for $i_1+\cdots+i_\ell=m$ with all $ i_a \ge 1$. Hence
$$Q_{m, \ell}=\frac{1}{\ell!}\sum_{\substack{i_1+\cdots+i_\ell=m \\ i_a \ge 1}}\prod_{a=1}^{\ell}(2i_a-1)!= \frac{1}{\ell!} \sum_{\substack{\sum_j jm_j=m \\ \sum_j m_j = \ell}} \frac{\ell!}{\prod_{j \ge 1}m_j!}\prod_{j \ge 1}((2j-1)!)^{m_j}=P_{m,\ell}.$$
For positive integers $a, b$, one has 
$$a!b! \le (a+b-1)!.$$ 
Applying this repeatedly gives 
$$\prod_{a=1}^{\ell}(2i_a-1)! \le(2m-2\ell+1)!.$$ 
Therefore
$$P_{m,\ell}=Q_{m, \ell}=\frac{1}{\ell!}\sum_{\substack{i_1+\cdots+i_\ell=m\\ i_a\ge1}}\prod_{a=1}^{\ell}(2i_a-1)! \le \frac{1}{\ell!}\sum_{\substack{i_1+\cdots+i_\ell=m\\ i_a\ge1}} (2m-2\ell+1)! = \frac{1}{\ell!} \binom{m-1}{\ell-1}(2m-2\ell+1)!.$$
Now, using \eqref{18}, we deduce that
$$\frac{T_\ell}{M} \le \frac{r^{\ell-1} \left(\frac{2d\sin(\frac{\pi}{d})}{\pi}\right)^{\ell}}{\frac{5}{4}(2m-1)!} \frac{1}{\ell!} \binom{m-1}{\ell-1}(2m-2\ell+1)!.$$
Define the right-hand side above to be
$$U_\ell=\frac{r^{\ell-1} \left(\frac{2d\sin(\frac{\pi}{d})}{\pi}\right)^{\ell}}{\frac{5}{4}(2m-1)!} \frac{1}{\ell!} \binom{m-1}{\ell-1}(2m-2\ell+1)!.$$
Then 
\begin{equation}
\label{19}
\frac{T_\ell}{M} \le U_\ell.
\end{equation} 
For $\ell=2$, we compute 
\begin{equation}
\label{22} 
U_2=\frac{r \left(\frac{2d\sin(\frac{\pi}{d})}{\pi}\right)^2}{5(2m-1)}
\end{equation} 
while, for $2 \le l \le m-1$,
$$\frac{U_{\ell+1}}{U_\ell}=\frac{r \left(\frac{2d\sin(\frac{\pi}{d})}{\pi}\right)}{2\ell (\ell+1)(2m-2\ell+1)}.$$
Since $\ell (\ell+1)(2m-2\ell+1) \ge 6(2m-3)$ for $2 \le \ell \le m-1$, we get that
\begin{equation}
\label{20}
\frac{U_{\ell+1}}{U_\ell} \le \frac{rd\sin(\frac{\pi}{d})}{6\pi(2m-3)}.
\end{equation}
Set 
$$x=\frac{6\pi(2m-3)}{rd\sin(\frac{\pi}{d})}$$
so that \eqref{20} gives
\begin{equation}
\label{21}
\frac{U_{\ell+1}}{U_\ell} \le \frac{1}{x}.
\end{equation} 
Observe that  $x > 1$: since $r < F(m)$ we have in particular that $r < 2m \le \frac{6\pi(2m-3)}{d\sin(\frac{\pi}{d})}$. Combining \eqref{21} with \eqref{19} we have
\begin{equation}
\label{23}
\sum_{\ell=2}^m \frac{T_\ell}{M} \le \sum_{\ell=2}^m U_\ell \le U_2 \sum_{\ell=2}^m \frac{1}{x^{\ell-2}} \le U_2 \sum_{k \ge 0} \frac{1}{x^k} = \frac{U_2 x}{x-1}<1
\end{equation} 
the latter because we have from \eqref{22} that
$$U_2=\frac{r \left(\frac{2d\sin(\frac{\pi}{d})}{\pi}\right)^2}{5(2m-1)}<\frac{x-1}{x}=\frac{6\pi(2m-3)-rd\sin(\frac{\pi}{d})}{6\pi(2m-3)}$$
if and only if
$$r < \frac{30\pi^2(2m-1)(2m-3)}{24d^2(2m-3)\sin^2(\frac{\pi}{d})+5\pi d(2m-1)\sin(\frac{\pi}{d})}=B(m)$$
and the above holds by the hypothesis $r < F(m)$. It follows from \eqref{23} that
$$\sum_{\ell=2}^m T_\ell<M.$$
In the expansion \eqref{15},
$$J_m=\sum_{\lambda\vdash m}\prod_{j \ge 1}\frac{A_j^{m_j}}{m_j!}=A_m+\sum_{\substack{\lambda\vdash m \\ \ell(\lambda) \ge 2}} \prod_{j \ge 1}\frac{A_j^{m_j}}{m_j!}=A_m+\sum\limits_{\ell \ge 2} \sum\limits_{\substack{\lambda\vdash m \\ \ell(\lambda) = \ell}} \prod_{j \ge 1}\frac{A_j^{m_j}}{m_j!}$$
we have that $|A_m|=M>0$, while
$$|\sum\limits_{\ell \ge 2} \sum\limits_{\substack{\lambda\vdash m \\ \ell(\lambda) = \ell}} \prod_{j \ge 1}\frac{A_j^{m_j}}{m_j!}| \le \sum\limits_{\ell \ge 2} \sum\limits_{\substack{\lambda\vdash m \\ \ell(\lambda) = \ell}} \prod_{j \ge 1}\frac{|A_j|^{m_j}}{m_j!}=\sum_{\ell=2}^m T_\ell<M$$
hence $J_m$ has the same sign as $A_m$ and \eqref{24} gives that $(-1)^{m-1}J_m>0$, hence $(-1)^{m-1}I_{2m}>0$ by \eqref{j}. This proves Claim \ref{finale}.
\end{proof}
\renewcommand{\proofname}{Proof}
To finish the proof of Theorem \ref{main}, we choose $m=\lfloor \frac{n}{2} \rfloor$ in Claim \ref{finale}. Suppose that $r<F(m)$. Then $c_{2m}(\F)H^{n-2m}=I_{2m} \ne 0$ by Claim \ref{finale}. Since $r < 2m$, this is a contradiction (see Remark \ref{dopor}) and the theorem is proved for $n$ odd or $n$ even and $X$ very general. Finally suppose that $n$ is even and $X$ has an Ulrich bundle $\E$ of rank $r$. Then its hyperplane section has odd dimension $n-1$ and again an Ulrich bundle, restriction of $\E$, of rank $r$. Hence, from the odd case, we get that $r \ge F(\frac{n-2}{2})$ and the proof is complete.
\end{proof}

\section{Numerical values}
\label{num}

We show here some calculations that allow to understand when $F(m)=2m$.
 
Using the fact that $\sin(x) \le x$ for $x \ge 0$ it is easily shown that $F(m)=2m$, that is $2m \le B(m)$, for $m \ge 21$.  We also note that $B(m) \ge \frac{3m}{2}$ for $3 \le m \le  20$.

For $3 \le m \le 20$, setting $x=\frac{\sin(\frac{\pi}{d})}{\frac{\pi}{d}}$, the inequality $2m \le B(m)$ is equivalent to
$$48m(2m-3)x^2+10m(2m-1)x-30(2m-1)(2m-3)\le 0$$
that is
$$0 \le x \le \frac{-5m(2m-1)+\sqrt{5m(2m-1)(2592-3461m+1162m^2)}}{48m(2m-3)}.$$
Setting
$$a(m)=\frac{-5m(2m-1)+\sqrt{5m(2m-1)(2592-3461m+1162m^2)}}{48m(2m-3)}$$
we see that
$$a(m)y-\sin(y) \ge 0$$
for $0 \le y \le \frac{\pi}{2}$ if and only if $y_0(m) \le y \le \frac{\pi}{2}$, where $y_0(m)$ is the only solution of the equation $a(m)y-\sin(y)=0$ for $0 < y \le \frac{\pi}{2}$. 

It follows that, for $3 \le m \le 20$, we have that $F(m)=2m$ if and only if $d \le \frac{\pi}{y_0(m)}$. In Table \ref{tab1} below we give the values of $y_0(m)$ and $d$ such that $F(m)=2m$. 

\begin{table}[htbp] % 'h' means here, 't' top, 'b' bottom, 'p' page
    \centering      % Centers the table horizontally on the page
    \caption{Values of $d$ for which $F(m)=2m$ }
    \label{tab1}
    % Adjusted column widths and used centered 'c' for clean math alignment
    \begin{tabular}{|c|c|l|} 
        \hline
        $m$ & $y_0(m)$ & value \\
        \hline
        3  & 0.93 & $d=3$       \\ \hline
        4  & 0.74 & $d \le 4$   \\ \hline
        5  & 0.62 & $d \le 4$   \\ \hline
        6  & 0.54 & $d \le 5$   \\ \hline
        7  & 0.48 & $d \le 6$   \\ \hline
        8  & 0.43 & $d \le 7$   \\ \hline
        9  & 0.38 & $d \le 8$   \\ \hline
        10 & 0.35 & $d \le 8$   \\ \hline
        11 & 0.31 & $d \le 9$   \\ \hline
        12 & 0.28 & $d \le 10$  \\ \hline
        13 & 0.25 & $d \le 12$  \\ \hline
        14 & 0.23 & $d \le 13$  \\ \hline
        15 & 0.20 & $d \le 15$  \\ \hline
        16 & 0.17 & $d \le 17$  \\ \hline
        17 & 0.15 & $d \le 20$  \\ \hline
        18 & 0.12 & $d \le 25$  \\ \hline
        19 & 0.09 & $d \le 34$  \\ \hline
        20 & 0.04 & $d \le 65$  \\ \hline
    \end{tabular}
\end{table}

\eject


\begin{thebibliography}{BKKMSU}

\bibitem[A]{ahl} L.~V.~Ahlfors.
\textit{Complex analysis. An introduction to the theory of analytic functions of one complex variable}. 
Internat. Ser. Pure Appl. Math. McGraw-Hill Book Co., New York, 1978. xi+331 pp.

%\bibitem[B1]{b1} A.~Beauville.
%\textit{Determinantal hypersurfaces}. 
%Michigan Math. J. \textbf{48} (2000), 39-64.

\bibitem[B]{b} A.~Beauville.
\textit{An introduction to Ulrich bundles}. 
Eur. J. Math. \textbf{4} (2018), no.~1, 26-36.

\bibitem[BGS]{bgs} R.~O.~Buchweitz, G.~M.~Greuel, F.~O.~Schreyer.
\textit{Cohen-Macaulay modules on hypersurface singularities. II}. 
Invent. Math. \textbf{88} (1987), no.~1, 165-182.

\bibitem[BES]{bes} M.~Bl\"aser, D.~Eisenbud, F.~O.~Schreyer.
\textit{Ulrich complexity}. 
Differential Geom. Appl. \textbf{55} (2017), 128-145.

%\bibitem[CFK]{cfk} C.~Ciliberto, F.~Flamini, A.~L.~Knutsen.
%\textit{Ulrich bundles on Del Pezzo threefolds}. 
%J. Algebra \textbf{634} (2023), 209-236.

%\bibitem[CH]{ch} M.~Casanellas, R.~Hartshorne.
%\textit{Stable Ulrich bundles}. With an appendix by F.~Geiss, F.-O.~Schreyer.
%Internat. J. Math. \textbf{23} (2012), no. 8, 1250083, 50 pp.

\bibitem[CMRPL]{cmrpl} L.~Costa, R.~M.~Mir\'o-Roig, J.~Pons-Llopis.
\textit{Ulrich bundles}.
De Gruyter Studies in Mathematics, \textbf{77}, De Gruyter 2021. 

\bibitem[E]{e} D.~Erman.
\textit{Matrix factorizations of generic polynomials}. 
Preprint arXiv:2112.08864.

\bibitem[ES]{es} D.~Eisenbud, F.-O.~Schreyer.
\textit{Resultants and Chow forms via exterior syzygies}. 
J. Amer. Math. Soc. \textbf{16} (2003), no. 3, 537-579.

\bibitem[F]{f} W.~Fulton.
\textit{Intersection theory}.
Ergeb. Math. Grenzgeb. (3), 2. Springer-Verlag, Berlin, 1998, xiv+470 pp.

\bibitem[FK]{fk} D.~Faenzi, Y.~Kim. 
\textit{Ulrich bundles on cubic fourfolds}. 
Comment. Math. Helv. \textbf{97} (2022), no.~4, 691-728.

\bibitem[HUB]{hub}
J.~Herzog, B.~Ulrich, J.~Backelin. 
\textit{Linear maximal Cohen-Macaulay modules over strict complete intersections}. 
J. Pure Appl. Algebra \textbf{71} (1991), no.~2-3, 187-202.

\bibitem[KS]{ks} Y.~Kim, F.-O.~Schreyer.
\textit{An explicit matrix factorization of cubic hypersurfaces of small dimension}. 
J. Pure Appl. Algebra \textbf{224} (2020), no.~8, 106346, 13 pp.

\bibitem[L]{laz} R.~Lazarsfeld. \textit{Positivity in algebraic geometry. II. Positivity for vector bundles, and multiplier ideals}.
Ergebnisse der Mathematik und ihrer Grenzgebiete. 3. Folge. A Series of Modern Surveys in Mathematics, \textbf{49}. Springer-Verlag, Berlin, 2004.

%\bibitem[LM]{lm} A.~F.~Lopez, R.~Mu\~{n}oz. 
%\textit{On the classification of non-big Ulrich vector bundles on surfaces and threefolds}.
%Internat. J. Math. \textbf{32} (2021), no.~14, Paper No.~2150111, 18 pp.

%\bibitem[LMS]{lms} A.~F.~Lopez, R.~Mu\~{n}oz, J.~C.~Sierra. 
%\textit{On the classification of non-big Ulrich vector bundles on fourfolds}.
%Ann. Sc. Norm. Super. Pisa Cl. Sci. (5) \textbf{26} (2025), no.~2, 707-755.

\bibitem[LR1]{lr1} A.~F.~Lopez, D.~Raychaudhury.
\textit{Ulrich subvarieties and the non-existence of low rank Ulrich bundles on complete intersections}.
Trans. Amer. Math. Soc. \textbf{379} (2026), no.\ 8, 5809-5846.

\bibitem[LR2]{lr2} A.~F.~Lopez, D.~Raychaudhury.
\textit{A lower bound on the Ulrich complexity of hypersurfaces}.
Preprint 2025, arXiv:2503.13396.

%\bibitem[LS]{ls} A.~F.~Lopez, J.~C.~Sierra. 
%\textit{A geometrical view of Ulrich vector bundles}.
%Int. Math. Res. Not. IMRN(2023), no.~11, 9754-9776.
 
\bibitem[RT1]{rt1} G.~V.~Ravindra, A.~Tripathi.
\textit{Rank 3 ACM bundles on general hypersurfaces in $\P^5$}. 
Adv. Math. \textbf{355} (2019), 106780, 33 pp.

\bibitem[RT2]{rt2} G.~V.~Ravindra, A.~Tripathi.
\textit{On the base case of a conjecture on ACM bundles over hypersurfaces}. 
Geom. Dedicata \textbf{216} (2022), no.~5, Paper No.~49, 10 pp.

\bibitem[S]{st} R.~P.~Stanley.
\textit{Enumerative combinatorics. Volume 1.}. 
Cambridge Stud. Adv. Math., \textbf{49}. Cambridge University Press, Cambridge, 2012. xiv+626 pp.

%\bibitem[T]{tr} A.~Tripathi.
%\textit{Rank 3 arithmetically Cohen-Macaulay bundles over hypersurfaces}. 
%J. Algebra \textbf{478} (2017), 1-11.

\end{thebibliography}
\end{document}